\documentclass[11p,leqno]{amsart}
\usepackage{amsmath}
\usepackage{amsfonts}
\usepackage{amssymb}
\usepackage{graphicx}
\usepackage{color}
\newtheorem{theorem}{Theorem}[section]
\newtheorem*{theorem*}{Theorem}
\newtheorem{lemma}{Lemma}[section]
\newtheorem{corollary}[theorem]{Corollary}

\def\R{\Bbb R}

\numberwithin{equation}{section}

\begin{document}
\title[isoperimetric comparison]{\bf Isoperimetric comparisons via viscosity}

\author{Lei Ni}

\address{Department of Mathematics, University of California at San Diego, La Jolla, CA 92093, USA}
\email{lni@math.ucsd.edu}
%    \thanks will become a 1st page footnote.

\author{ Kui Wang}

%    Information for second author

\address{School of Mathematic Sciences, Fudan University, Shanghai, 200433, China}
\email{kuiwang09@fudan.edu.cn}

%    General info
\date{}

\maketitle

\begin{abstract} Viscosity solutions are suitable notions in the study of nonlinear PDEs justified by  estimates established via the maximum principle or the comparison principle. Here we prove that the isoperimetric profile functions of  Riemannian manifolds with Ricci lower bound are viscosity super-solutions of some nonlinear differential equations. From these one can derive the isoperimetric inequalities of L\'evy-Gromov and B\'erard-Besson-Gallot, as well as a upper bound of Morgan-Johnson.
\end{abstract}

\section{Introduction}

Viscosity solutions are solutions with usually less regularity.  However this flexibility is   important in the development and the study of nonlinear PDEs. For motivations, examples and techniques  see  e.g. \cite{Evans, Han-Lin}. One particular advantage of the concept is that it allows effective uses of the comparison principle so that crucial estimates can be established for  existence and uniqueness even though the viscosity solutions are a much broader class of solutions.

Let $(M^n, g)$ be a compact Riemannian manifold. The  isoperimetric profile function is defined as follows. For any $\beta\in (0, 1)$, consider smooth regions $\Omega\subset M$ such that
the volume $|\Omega|$ satisfying $|\Omega|=\beta |M|$, and let
$h_1(\beta, g)=\inf_{\Omega} \frac{|\partial \Omega|}{|M|}$. Here $|\partial \Omega|$ denotes the $n-1$-dimensional area of $\partial \Omega$, and the infimum is taken over all $\Omega$ satisfying the volume constraint. The profile function  generally is  not smooth, but continuous.

In this short note we shall prove that isoperimetric profile functions are  viscosity super-solutions of some nonlinear differential equations. From these one can derive the isoperimetric inequalities of L\'evy-Gromov \cite{Gromov} and B\'erard-Besson-Gallot \cite{BBG}, as well as some new comparison results.

The consideration here  is motivated by a  paper of Andrews-Bryan \cite{AB}, where the authors established some comparisons for the isoperimetric profile functions for metrics on a two-sphere deformed by the Ricci flow motivated by the earlier work of Hamilton \cite{H}. The comparison here is for static metrics satisfying some conditions on the Ricci curvature. There are two differential equations involved (see Theorem \ref{thm-main1} in Section 2 and Theorem \ref{thm-main3} in Section 4). One is of second order, which can be viewed as a stability result for the isoperimetric profile function. This equation has more or less been shown previously in the works of \cite{Bayle} \cite{MorganJohnson} (see \cite{Morgan} and \cite{Bayle-Ros} as well).  The other is of first order, which can be viewed as a Hamilton-Jacobi type equation. The consideration also leads to an alternate  proof of a  comparison theorem in Section 3, which was originally proved in \cite{MorganJohnson}.  We hope to investigate further in the future the application of this approach to the study of the isoperimetric profile functions. The interested reader should consult the survey article \cite{Andrews} and the references therein for related and more recent developments.

\section{Isoperimetric profile function as a viscosity supersolution}

Let $(M^n, g)$ be a compact Riemannian manifold. The  isoperimetric profile function is defined as follows. For any $\beta\in (0, 1)$, consider smooth region $\Omega\subset M$ such that
its volume $|\Omega|$ satisfying $|\Omega|=\beta |M|$, let
$h_1(\beta, g)=\inf_{\Omega} \frac{|\partial \Omega|}{|M|}$. Here $|\partial \Omega|$ denotes the $n-1$-dimensional area of $\partial \Omega$, and the infimum is taken for all $\Omega$ satisfying the volume constraint. It is known (cf. Chapter VI of \cite{Sakai}) that $h_1(\beta, g)$ is continuous (in fact H\"older continuous), satisfying the symmetry $h_1(\beta, g)=h_1(1-\beta, g)$. Moreover it has the asymptotics (cf. Proposition 1.3 of Chapter VI in \cite{Sakai}):
\begin{equation}\label{eq-asy}
\lim_{\beta\to 0} \frac{h_1(\beta, g)}{\beta^{\frac{n-1}{n}}}=n\frac{\sigma_n^{1/n}}{|M|^{1/n}},
\end{equation}
where $\sigma_n$ denotes the volume of the unit ball in the Euclidean space $\R^n$.

The isoperimetric inequality of L\'evy-Gromov \cite{Gromov} asserts the the following:
\begin{theorem}[L\'evy-Gromov]\label{thm-gromov} Assume that the Ricci curvature of $(M, g)$, $Ric_g\ge (n-1)\kappa g$ for some $\kappa>0$. Then
$$
h_1(\beta, g)\ge h_1(\beta, g_\kappa)
$$
where $(M_k, g_k)$ is the space form of constant sectional curvature $k$.
\end{theorem}

We prove the following result which  implies the above inequality via a maximum principle for viscosity solutions. This argument avoids the estimate of Heintze-Karcher.

\begin{theorem}\label{thm-main1} Assume that the Ricci curvature of the compact manifold $(M^n, g)$, $Ric_g\ge \kappa(n-1) g$. The isoperimetric profile function  $h_1(\beta, g)$, as a function of $\beta$, is a positive viscosity supersolution (over $(0, 1)$) of the differential equation:
\begin{equation}\label{eq:vis}
-\psi'' \psi = (n-1)\left(k+\left(\frac{\psi'}{n-1}\right)^2\right).
\end{equation}
\end{theorem}

Before we prove the result we first derive Theorem \ref{thm-gromov} from the above. First observe that
$h_1(\beta, g_1)$ is a smooth solution to (\ref{eq:vis}) on $(0, 1)$. By scaling it suffices to prove it for $k=1$. Assume that the claimed estimate in Theorem \ref{thm-gromov}  fails. Then by the asymptotics and the symmetry,  there exists $\beta_0\in (0, 1)$ such that $h_1(\beta, g)-h_1(\beta, g_1)$ attains its negative minimum. Now  in a small neighborhood of $\beta_0$, there exists a smooth $\varphi(\beta)$ such that $\varphi(\beta)\le h_1(\beta, g)$ and $\varphi(\beta_0)=h_1(\beta_0, g)$. This support function can be constructed easily from $h_1(\beta, g_1)$ which is smooth. Moreover we have that
\begin{equation}\label{eq:vis2}
-\varphi'' \varphi\ge  (n-1)\left(1+\left(\frac{\varphi'}{n-1}\right)^2\right).
\end{equation}
On the other hand by the above $\varphi(\beta)-h_1(\beta, g_1)$ attains a local negative minimum at $\beta_0$. Hence we have that $h_1(\beta_0, g_1)>\varphi(\beta_0)>0$,  $\varphi'(\beta_0)=h_1'(\beta_0, g_1)$ and $\varphi''(\beta_0)\ge h_1''(\beta_0, g_1)$. By writing $h_1(\beta, g_1)$ as $h_{1, g_1}(\beta)$, this implies that at $\beta_0$,
\begin{eqnarray*}
-\varphi(\beta) \varphi''(\beta) &\le& -\varphi(\beta) h_{1, g_1}''(\beta)\\
&=& \frac{\varphi(\beta)}{h_{1, g_1}(\beta)} (-h_{1, g_1}(\beta) h_{1, g_1}''(\beta))\\
&=& \frac{\varphi(\beta)}{h_{1, g_1}(\beta)} \cdot (n-1)\left(1+\left(\frac{\varphi'(\beta)}{n-1}\right)^2\right)
\end{eqnarray*}
which is a contradiction to (\ref{eq:vis2}), by noting that $ \frac{\varphi(\beta_0)}{h_{1, g_1}(\beta_0)}<1$.

Now we prove Theorem \ref{thm-main1}. By the definition we  need to verify that for any $\beta_0$, and a small neighborhood $U$ of it, a smooth function  $0<\psi(\beta)\le h_1(\beta, g)$ in $U$ with $\psi(\beta_0)= h_1(\beta_0, g)$, the equation (\ref{eq:vis2}) holds at $\beta=\beta_0$. Let $\Omega$ be the domain  minimizing $|\partial \Omega|$ with $|\Omega|=\beta_0 |M|$. Let $\partial \Omega$ denote the boundary of $\Omega$. By the regularity theorem \cite{Simon}, $\partial \Omega$ is a smooth hypersurface except a singular set of Hausdorff codimension $7$. The mean curvature of $N$, the smooth part,  is defined  and is a constant. For a small region $D$ of $N$, we may consider the variation given by $\exp_{x} (t\eta(x)  \nu(x))$ with $\nu$ being the unit outward normal, $\eta$ being a  function supported in $D$. Let $N_t$ be this {\it variation} of N and let $\Omega_t$ be the domain bounded by $N_t$ (together with the irregular part of $\partial \Omega$, which not altered). Recall that $\exp_N((x, t))=\exp_x(t\nu(x))$. Simple calculation shows that if $J(\exp_N)|_{(x, s)}=a(x, s)$ with $a(x, 0)=1$,
\begin{eqnarray*}
|\Omega_t|&=&|\Omega|+ \int_D \int_0^{t\eta} a(x, s)\, ds d\mu_{g_{N}}, \\
\left.\frac{d}{dt}|\Omega_t|\right|_{t=0}&=&\int_N \eta \, d\mu_{g_N}.
\end{eqnarray*}
Recall that the 1st variation formula for the submanifolds  also gives
$$
\left.\frac{d}{dt}|N_t|\right|_{t=0}=(n-1)\int_N \eta H \, d\mu_{g_{N}}.
$$
Let $\beta(t)=\frac{|\Omega_t|}{|M|}$. It is easy to see that $\psi(\beta(t))\le \frac{|N_t|}{|M|}$ and $\psi(\beta(0))=\psi(\beta_0)=\frac{|N_0|}{|M|}.$ Now let $F(t)=\frac{|N_t|}{|M|}-\psi(\beta(t))$ which attains a local minimum at $t=0$. The first variation formula yields that
\begin{equation}\label{eq:1st}
(n-1)H=\psi'(\beta_0)
\end{equation}
Note that $\frac{d}{dt}|_{t=0} \psi(\beta(t))=\psi'(\beta_0)\frac{1}{|M|}\int_N \eta$.
The fact that $F''(0)\ge 0$ and the second variational formula (cf. page 8 of \cite{Li}) yields at $t=0$ ($\beta=\beta_0$)
\begin{equation}\label{eq:2nd1}
\frac{1}{|M|}\int_N |\nabla \eta|^2+(\eta (n-1)H)^2-\eta^2h_{ij}^2-\eta^2 \operatorname{Ric}(\nu, \nu)\ge \psi'' \left(\frac{1}{M}\int_N \eta\right)^2 +\psi' \frac{(n-1)H}{|M|}\int_N \eta.
\end{equation}
The smallness of the singular set allows $\eta  =1$, via approximations. Hence we have that  for $\beta=\beta_0$
\begin{equation}\label{eq:2nd2}
-\psi'' \psi^2 \ge  (n-1) \psi \left(\frac{\psi'}{n-1}\right)^2 +\frac{1}{|M|}\int_N \operatorname{Ric}(\nu, \nu).
\end{equation}
This proves the claimed differential inequality by cancelation and using $\operatorname{Ric}(\nu, \nu)\ge k(n-1)$.

Consequences include the following result for the case of $\kappa=0$ and $\kappa=-1$.

 \begin{corollary}\label{coro-1} (i) Assume that the Ricci curvature of $(M, g)$, $Ric_g\ge 0$. The isoperimetric profile function  $h_1(\beta, g)$, as a function of $\beta$, is a positive super solution of the differential equation:
\begin{equation}\label{eq:visk0}
-\psi'' \psi = (n-1)\left(\frac{\psi'}{n-1}\right)^2.
\end{equation}

 (ii) Assume that the Ricci curvature of $(M, g)$, $Ric_g\ge -(n-1) g$. The isoperimetric profile function  $h_1(\beta, g)$, as a function of $\beta$, is a positive super solution of the differential equation:
\begin{equation}\label{eq:visk-1}
-\psi'' \psi = (n-1)\left(-1+\left(\frac{\psi'}{n-1}\right)^2\right).
\end{equation}

 \end{corollary}

\section{Comparisons from above on manifolds with  Ricci  lower bound}

Motivated with the consideration of the last section we consider $(M, g)$ with $\operatorname{Ric}_g \ge (n-1)\kappa$, where $\kappa$  is a constant, but not necessarily satisfying $\kappa> 0$. This of courses allows manifolds with infinity volume. Now we define  $h_2(\beta, g)$, another profile function which is natural for this setting, as $\inf |\partial \Omega|$ among all $\Omega$ such that $|\Omega|=\beta$. Clearly $h_2(\beta, g)$ is now defined for $(0, |M|)$. When $|M|<\infty$, $h_2(\beta, g)=|M|h_1(\frac{\beta}{|M|}, g)$.

The following comparison result holds for the profile function $h_2(\beta, g)$.

\begin{theorem}\label{thm-main2} Let $(M, g)$ be a complete Riemannian manifold with $\operatorname{Ric}_g \ge (n-1)\kappa$. Then
$$
h_2(\beta, g)\le h_2(\beta, g_\kappa)
$$
for $\beta\in (0, |M|)$. If the equality ever holds somewhere, $(M, g)$ must be isometric to the space form $(M, g_\kappa)$.
\end{theorem}

Note that the famous Cartan-Hadamard conjecture asserts the opposite estimate if $(M, g)$ is a Cartan-Hadamard manifold with the sectional curvature $K_M \le \kappa\le 0$. The result is an analogue of the eigenvalue comparison result of Cheng \cite{Cheng}. This result was first proved in \cite{MorganJohnson}. Below is an alternate argument.

This profile function satisfies the scaling law
$h_2(\beta, cg)=c^{\frac{n-1}{2}}h_2(c^{-\frac{n}{2}}\beta, g)$.  Hence it suffices to prove for cases $\kappa=-1, 0, 1$. For the proof we need the following simple lemma (one can find its proof for example in \cite{Ni}).

\begin{lemma}\label{lemma31}   Let  $\rho(t)$ be a continuous function on $[0, b]$.  Assume  that $\rho(0)\le 0$ and there exist some positive constants $\epsilon, C$ such that   $D^{-}\rho \le C \rho$, whenever $0<\rho(t)\le \epsilon$. Then $\rho(b)\le 0$. The same result holds if $D^{-}$ is replaced by $D^{+}$, $D_{-}$ or $D_{+}$.\end{lemma}

To prove Theorem \ref{thm-main2}, let $p\in M$ be a fixed point and introduce $I_p(\beta, g)=|\partial B_p(r)|$ with $|B_p(r)|=\beta$. Clearly $h_2(\beta, g)\le I_p(\beta, g)$ while $h_2(\beta, g_\kappa)=I_{\bar{p}}(\beta, g_\kappa)$ where $\bar{p}\in M_\kappa$ is a fixed point in the space form $M_\kappa$. The claimed result follows if we can establish that $I_p(\beta, g)\le h_2(\beta, g_\kappa)$.
Let $f(\beta)=I_p(\beta, g)-I_{\bar{p}}(\beta, g_\kappa)$. Since $f(0)=0$ it suffices to show that
$f'\le 0$ by Lemma \ref{lemma31}.  Let $B_{\bar{p}}(\bar{r})$ be the ball in $M_k$ such that $|B_{\bar{p}}(\bar{r})|=\beta$. Now the direct calculations shows that
$$
I_p'= \frac{n-1}{|\partial B_p(r)|}\int_{\partial B_p(r)} H(r, \theta); \quad
I_{\bar{p}}' = \frac{n-1}{|\partial B_{\bar{p}}(\bar{r})|}\int_{\partial B_{\bar{p}}(\bar{r})} \bar{H}(\bar{r})=(n-1)\bar{H}(\bar{r}).
$$
Here denotes the mean curvature of $\partial B_p(r)$ in terms of the polar coordinate and $\bar{H}(\bar{r})$ be the mean curvature of $\partial B_{\bar{p}}(\bar{r})$ in the space form $M_\kappa$. By the volume comparison theorem $|B_p(r)|\le |B_{\bar{p}}(r)|$, which implies that $\bar{r}\le r$ since
$\beta=|B_p(r)|=|B_{\bar{p}}(\bar{r})|$.  Also by the Laplacian comparison theorem $H(r, \theta)\le \bar{H}(r)$. Noting that $\bar{H}(s)$ (mean curvature of the spheres in $M_\kappa$) is a monotone non-increasing function. We then have
\begin{eqnarray*}
\frac{n-1}{|\partial B_p(r)|}\int_{\partial B_p(r)} H(r, \theta)&\le& (n-1) \bar{H}(r)\\
&\le & (n-1) \bar{H}(\bar{r}).
\end{eqnarray*}
This proves $f'\le 0$, hence the claimed inequality in Theorem \ref{thm-main2}. The equality case follows from the equality in the volume comparison (applying to balls with varying centers).

Combining Theorem \ref{thm-gromov}   and Theorem \ref{thm-main2} we have the two sided bounds below.

\begin{corollary}\label{coro-twosides}Assume that the Ricci curvature of $(M, g)$, $Ric_g\ge (n-1) g$. Then
$$
h_1(\beta, g_1)\le h_1(\beta, g)  \le  \frac{|\mathbb{S}^n|}{|M|} \cdot \, h_1\left(\frac{|M|}{|\mathbb{S}^n|}\beta, g_1\right) .$$
\end{corollary}

The scaling relation between $h_2(\beta, g)$ and $h_1(\beta, g)$ yields that $h_2(\beta, g)$ also satisfying Theorem \ref{thm-main1} when $\kappa=1$. Namely on $(0, |M|)$, $h_2(\beta, g)$ is a viscosity super-solution of (\ref{eq:vis}).
This implies the following corollary.

\begin{corollary} Let $(M, g)$ be a compact Riemannian manifold with $\operatorname{Ric}_g \ge (n-1)\kappa$. Then if $\kappa> 0$,
$
\frac{h_2(\beta, g)}{h_2(\beta, g_\kappa)}
$ is a monotone non-increasing function on $(0, |M|)$.
\end{corollary}

\begin{proof} For $\kappa>0$, without the loss of generality we assume $\kappa=1$. Now  notice that $h_2(\beta, g_\kappa)$ is a smooth solution of (\ref{eq:vis}) and $h_2(\beta, g)$ is a viscosity super solution of (\ref{eq:vis}). By Theorem \ref{thm-main2} we have  that $
\frac{h_2(\beta, g)}{h_2(\beta, g_\kappa)}\le 1
$. If the claimed result does not hold, then one can find $0<\beta_1< \beta_2<|M|$ such that
$$
\frac{h_2(\beta_1, g)}{h_2(\beta_1, g_\kappa)}
<
\frac{h_2(\beta_2, g)}{h_2(\beta_2, g_\kappa)}.
$$
Then $\frac{h_2(\beta, g)}{h_2(\beta_1, g_\kappa)}$ achieves  minimum in $(0, \beta_2)$ which is strictly smaller than $1$, by Theorem \ref{thm-main2}. Now one can repeat the argument in  the proof of L\'evy-Gromov isoperimetric estimate to arrive a contradiction! Precisely, assume that the minimum is attained at $\beta_0$, for a neighborhood $U$ and a support function $\psi>0$ of $h_2(\beta, g)$, we have that $\frac{\psi(\beta)}{h_2(\beta_1, g_1)}$ attains a local minimum at $\beta_0$. Then at $\beta_0$
\begin{eqnarray}
&\quad&
\frac{\psi'(\beta)}{h_2'(\beta, g_1)}= \frac{\psi(\beta)}{h_2(\beta, g_1)}\doteqdot \lambda <1 \label{eq:q1}  \\
&\quad& \frac{\psi''(\beta)}{h_2(\beta, g_1)}-\frac{\psi(\beta) h_2''(\beta, g_1)}{h^2_2(\beta, g_1)}\ge 0. \label{eq:q2}
\end{eqnarray}
Combining (\ref{eq:q1}), (\ref{eq:q2}) and that $\psi$ satisfying (\ref{eq:vis2}) we have
\begin{eqnarray*}
(n-1)\left(1+\left(\frac{\psi'}{n-1}\right)^2\right)&\le& -\psi'' \psi \\
&\le& \lambda^2 \left(-h_2 h_2''\right)\\
&=& (n-1)\left(\lambda^2 +\left(\frac{\lambda h_2'}{n-1}\right)^2\right)\\
&=&  (n-1)\left(\lambda^2+\left(\frac{\psi'}{n-1}\right)^2\right).
\end{eqnarray*}
This is a contradiction since $\lambda<1$.
\end{proof}

\section{B\'erard-Besson-Gallot comparison via the viscosity}

Here using the ideas from the Section 2 we derive some first order equation satisfied by the profile function. This together with the maximum principle argument implies the improved lower estimate of B\'erard-Besson-Gallot \cite{BBG}. As in \cite{BBG} we need to use the Heintze-Karcher estimates (cf. Theorem 3.8 of Chapter IV in \cite{Sakai}), unlike in the case for L\'evy-Gromov's estimate.

For manifold $(M, g)$ with $\operatorname{Ric}(g)\ge (n-1)\kappa $, let $$
s_\kappa(t)=\left\{ \begin{matrix} \frac{1}{\sqrt{\kappa}}\sin \sqrt{\kappa} t, & \, \kappa>0,\cr
                                   t, &\, \kappa=0, \cr
                                   \frac{1}{\sqrt{|\kappa|}}\sinh \sqrt{|\kappa|}t, &\, \kappa<0;\end{matrix}\right.
\quad \quad
c_\kappa(t)=\left\{ \begin{matrix} \cos \sqrt{\kappa} t, & \, \kappa>0,\cr
                                   1, &\, \kappa=0, \cr
                                   \cosh \sqrt{|\kappa|}t, &\, \kappa<0.\end{matrix}\right.
$$
Let $d$ denote the diameter of the manifold. Since for the consideration in this section, one only gets the sharp result for $\kappa>0$, we shall focus on this case first. Define
$$
\lambda^\kappa_{n,d} = \int_{-\frac{d}{2}}^{\frac{d}{2}} c^{n-1}_\kappa(t)\, dt=\frac{1}{\sqrt{\kappa}}\int_{-\frac{\sqrt{\kappa}d}{2}}^{\frac{\sqrt{\kappa}d}{2}} \cos^{n-1}(t)\, dt.
$$

\begin{theorem}\label{thm-main3} Assume that the Ricci curvature of $(M, g)$, $Ric_g\ge (n-1)\kappa g$, with $\kappa>0$. Let $d$ be the diameter of $(M, g)$. The isoperimetric profile function  $h_1(\beta, g)$, as a function of $\beta$, is a positive viscosity supersolution of the differential equation:
\begin{equation}\label{eq:vis-3}
\psi \left(1+\frac{1}{\kappa}\left(\frac{\psi'}{n-1}\right)^{2}\right)^{\frac{n-1}{2}}= \frac{1}{\lambda^\kappa_{n, d}}.
\end{equation}
\end{theorem}
\begin{proof}  The derivation follows essentially the argument in  \cite{BBG}. By scaling invariance of the result, we may assume that $\kappa=1$. By the definition we  need to verify that for any $\beta_0$, and a small neighborhood $U$ of it, a smooth function  $0<\psi(\beta)\le h_1(\beta, g)$ in $U$ with $\psi(\beta_0)= h_1(\beta_0, g)$, the inequality
  \begin{equation}\label{eq:vis-4}
 \psi \left(1+\left(\frac{\psi'}{n-1}\right)^{2}\right)^{\frac{n-1}{2}}\ge \frac{1}{\lambda^1_{n, d}}
  \end{equation}
  holds at $\beta=\beta_0$. Let $\Omega$ be the domain  minimizing $|\partial \Omega|$ with $|\Omega|=\beta_0 |M|$. Let $\partial \Omega$ denote the boundary of $\Omega$. Let $\eta, D,  N_t, \Omega_t$ be as those quantities in the proof of Theorem \ref{thm-main1}. Similar as before we have that
  \begin{equation}\label{eq:first-var}
  (n-1)H=\psi'(\beta_0)
  \end{equation}
where $H$ is the mean curvature of the regular part of $\partial \Omega$. As in \cite{BBG}, let $$r_0=\max\{\operatorname{dist}(x, \partial \Omega)\,|\, x\in \Omega\}.$$ It is easy to see  that
$$
r_1\doteqdot \max\{\operatorname{dist}(x, \partial \Omega)\, |\, x\in M\setminus \Omega\}\le d-r_0.
$$
In fact, for any $x_1\in \Omega$ and $x_2\in M\setminus \Omega$, let $\gamma(s)$ be the minimum geodesic joining from $x_1=\gamma(0)$ to $x_2=\gamma(l)$.
Hence $l=L(\gamma)\le d$. On the other hand, assume that $s_1>0$ is the  first time $\gamma(s)\in \partial \Omega$ and $s_2$ is the last time $\gamma(s)\in \partial \Omega$. Then
\begin{eqnarray*}
d\ge l &=& L(\gamma)
\ge s_1 +l-s_2
\ge r_0+r_1.
\end{eqnarray*}
Now we use Heintze-Karcher's estimate (cf. Theorem 3.8 of Chapter IV in \cite{Sakai}) to conclude that
\begin{eqnarray*}
|\Omega| &\le& |N| \int_0^{r_0} \left(\cos t-H\sin t\right)_{+}^{n-1}\, dt;\\
|M\setminus \Omega| &\le& |N| \int_0^{d-r_0} \left(\cos t+H\sin t\right)_{+}^{n-1}\, dt.
\end{eqnarray*}
Putting them together we have that
\begin{equation}\label{eq-104}
1\le \psi(\beta_0)\int_{r_0-d}^{r_0} \left(\cos t-H\sin t\right)_{+}^{n-1}\, dt.
\end{equation}
Writing $\cos \theta_0= \frac{1}{\sqrt{1+H^2}}, \sin \theta_0 =\frac{H}{\sqrt{1+H^2}}$,  we have  that
\begin{eqnarray*}
1&\le& \psi(\beta_0) (1+H^2)^{\frac{n-1}{2}} \int_{r_0-d}^{r_0} \left[\cos (t+\theta_0)\right]_+^{n-1}\, dt\\
&\le& \psi(\beta_0)\left(1+\left(\frac{\psi'}{n-1}\right)^2\right)^{\frac{n-1}{2}}\int_{-\frac{d}{2}}^{\frac{d}{2}}\cos^{n-1} t\, dt.
\end{eqnarray*}
This implies the claimed result. \end{proof}

A direct consequence is the B\'erard-Besson-Gallot's estimate. By scaling, without the loss of generality we may assume $k=1$. It is well-known that the diameter of the manifold $d$ is bounded from the above by $\pi$. Define
$$ \gamma_n\doteqdot\int_{-\frac{\pi}{2}}^{\frac{\pi}{2}} \cos ^{n-1} t\, dt, \quad \alpha(n, d)\doteqdot \left(\frac{\gamma_n}{\lambda^1_{n, d}}\right)^{\frac{1}{n}}.
$$
\begin{theorem}[B\'erard-Besson-Gallot]\label{thm-bbg} Let $(M^n, g)$ be a compact Riemannian manifold with $\operatorname{Ric}\ge (n-1)g$, $\gamma_n, \lambda^1_{n, d}, \alpha$ be as the above. Then
\begin{equation}\label{est-bbg}
h_1(\beta, g)\ge \alpha \cdot h_1(\beta, g_1)
\end{equation}
\end{theorem}
  This improves L\'evy-Gromov Theorem \ref{thm-gromov} since $\alpha\ge 1$  with the equality  if and only if $M$  is isometric to the round sphere.

To prove (\ref{est-bbg}) we first observe that $h_1(\beta, g_1)$ is a solution of (\ref{eq:vis-3}) with $d=\frac{\pi}{2}$. For simplicity we denote $h_1(\beta, g_1)$ by $\varphi(\beta)$. Hence $\varphi$ satisfies
\begin{equation}\label{eq-vis-e}
\varphi \left(1+\left(\frac{\varphi'}{n-1}\right)^{2}\right)^{\frac{n-1}{2}}= \frac{1}{\gamma_n}
\end{equation}

Assume that the claimed result fails. By the asymptotics we conclude that $\frac{h_1(\beta, g)}{\alpha \varphi( \beta)}$ attains the minimum $\lambda<1$ at some interior point $\beta_0$. At this point apply Theorem \ref{thm-main3} to the support function $\psi(\beta)>0$ with $\psi(\beta_0)=\lambda \alpha \varphi(\beta_0)$ we conclude that at $\beta_0$
$$
\psi'=\lambda \alpha \varphi'
$$
 and
 \begin{eqnarray*}
 \lambda \alpha \varphi \left(1+\left(\frac{\lambda \alpha \varphi'}{n-1}\right)^2\right)^{\frac{n-1}{2}}&=& \psi \left(1+\left(\frac{\psi'}{n-1}\right)^{2}\right)^{\frac{n-1}{2}}\\
 &\ge & \frac{\gamma_n}{\lambda_n} \varphi \left(1+\left(\frac{\varphi'}{n-1}\right)^{2}\right)^{\frac{n-1}{2}}\\
 &=& (\alpha \varphi) \left(\alpha^2 +\left(\frac{\alpha \varphi'}{n-1}\right)^2\right)^{\frac{n-1}{2}}.
 \end{eqnarray*}
The above estimate yields a contradiction since $\lambda<1$,  and $\alpha\ge 1$, $\varphi(\beta_0)>0$.

When $\kappa= 0$, a similar argument proves the following result.

\begin{corollary} Let $(M^n, g)$ be a compact Riemannian manifold with $\operatorname{Ric}\ge 0$. Let
$$\lambda^0_{n, d}\doteqdot \int_0^d (1+t^2)^{\frac{n-1}{2}}\, dt, \quad \alpha'(n, d)\doteqdot \left(\frac{\gamma_n}{\lambda^0_{n, d}}\right)^{\frac{1}{n}}.$$
Then $h_1(\beta, g)$ is a viscosity positive super-solution of the equation
\begin{equation}\label{eq:vis-k0}
\psi \left(1+\left(\frac{\psi'}{n-1}\right)^{2}\right)^{\frac{n-1}{2}}= \frac{1}{\lambda^0_{n, d}}.
\end{equation}
In particular, it implies that $h_1(\beta, g)\ge \alpha'(n, d)\cdot h_1(\beta, g_1)$.
\end{corollary}

\section*{Acknowledgments.} { The first author's research is partially supported by a NSF grant DMS-1401500. We would like to thank Frank Morgan for pointing out the related earlier works of his and Bayle's. }

\bibliographystyle{amsalpha}

\end{document}